\documentclass[11pt,a4paper,notitlepage]{article}
\usepackage{amsmath}
\usepackage[affil-it]{authblk}
\usepackage{amsthm}
\usepackage{amssymb}
\usepackage{enumerate}
\usepackage[page]{appendix}
\usepackage{geometry}
\usepackage{tabularx}
\usepackage{multirow}
\newtheorem{Lemma}      {Lemma} [section]
\newtheorem{Theorem}    [Lemma] {Theorem}

\newtheorem{Proposition}[Lemma] {Proposition}

\numberwithin{equation}{section}

\newcommand{\Stab}{\mathop{\mathrm{Stab}}}

\newcommand{\Sym}{\mathop{\mathrm{Sym}}}

\newcommand{\Alt}{\mathop{\mathrm{Alt}}}
\newcommand{\bl}{\mathop{\mathrm{bl}}}

\newcommand{\ws}{\widetilde}
\numberwithin{equation}{section}
\begin{document}
\title{Comparing the order and the minimal number of generators of a transitive permutation group}
\author{Gareth Tracey%
\thanks{Electronic address: \texttt{G.M.Tracey@warwick.ac.uk}}} 
\affil{Mathematics Institute, University of Warwick,\\Coventry CV4 7AL, United Kingdom}
\date{January 1, 2016}
\maketitle
\begin{abstract} We prove that if $G$ is a transitive permutation group, then $d(G)\log{|G|}/n^{2}$ tends to $0$ as $n$ tends to $\infty$.\end{abstract}
\section{Introduction and preliminary results}
\subsection{Introduction}
The purpose of this note is to prove the following:
\begin{Theorem}\label{LimitTheorem} Let $f(n)$ be the maximum of $d(G)\log{|G|}$ as $G$ runs over the transitive permutation groups of degree $n$. Then $f(n)/n^{2}$ tends to $0$ as $n$ tends to $\infty$.\end{Theorem}
The notation $d(G)$ denotes the minimal number of generators for $G$. All logs are to the base $2$ unless otherwise stated. 

Our strategy for the proof of the theorem will be to bound $d(G)\log{|G|}$, for a fixed transitive group $G$, in terms of the degrees of a set of primitive components for $G$, and another invariant $\bl_{W}(G)$ of $G$ (we define $\bl_{W}(G)$, and the term \emph{primitive components} in Section 1.3). The key result in this direction is Lemma \ref{MainLemma}, which we prove in Section 2. Section 1.2 and 1.3 contain required preliminary and elementary results, while Section 3 is reserved for the proof of Theorem \ref{LimitTheorem}.  

\subsection{Minimal generator numbers in transitive permutation groups}
We begin with a bound on the minimal number of generators for a transitive permutation group of degree $n$.
\begin{Theorem}[{\bf\cite{GT}, Corollary 1.2}]\label{TransTheoremCor1} Let $G$ be a transitive permutation group of degree $n\ge 2$. Then $d(G)\le \lfloor c_{1}n/\sqrt{\log{n}}\rfloor$, where $c_{1}:=0.920584\hdots$.\end{Theorem}

When $G$ is primitive, the bound for $d(G)$ is much sharper:
\begin{Theorem}[{\bf\cite{derek}, Theorem 1.1}]\label{derekthm} Let $H$ be a subnormal subgroup of a primitive permutation group of degree $r$. Then $d(H)\le\lfloor\log{r}\rfloor$, except that $d(H)=2$ when $r=3$ and $H\cong S_{3}$.\end{Theorem}

In order to prove Theorem \ref{LimitTheorem}, we will also need to improve the upper bound for $d(G)$ when $G$ is imprimitive. Before stating the results we need, we introduce a definition: for a finite group $R$, and a transitive permutation group $S\le \Sym{(s)}$, consider the wreath product $W=R\wr S$. Let $B:=R_{1}\times R_{2}\times\hdots\times R_{s}$ be the base group of $W$, and let $\pi:W\rightarrow \Sym{(s)}$ be the projection onto the top group. Also, since $N_{W}(R_{i})\cong R_{i}\times (R\wr \Stab_{s}(i))$, we may consider the projection maps $\rho_{i}:N_{W}(R_{i})\rightarrow R_{i}$. Then say that a subgroup $G$ of $W$ is \emph{large} if
\begin{enumerate}[(1)]
\item $\pi(G)=S$, and;
\item $\rho_{i}(N_{G}(R_{i}))=R_{i}$ for all $i$.\end{enumerate}

The results we need, both from \cite{GT}, can now be stated as follows (here, $(t)_{p}$ denotes the $p$-part of the positive integer $t$):
\begin{Lemma}[{\bf\cite{GT}, Lemma 4.1 part (ii)}] \label{pmodlemma} Let $G$ be a finite group, $H$ a subgroup of $G$ of index $n\ge 2$, $F$ a field of characteristic $p>0$, and $V$ a $H$-module of dimension $a$ over $F$. Also, let $T$ be a soluble subgroup of $G$, and let $t_{i}$, for $1\le i\le m$, denote the lengths of the orbits of $T$ on the set of right cosets of $H$ in $G$. Let $S$ be a submodule of the induced module $V\uparrow^{G}$. Then $d_{G}(S)\le a\sum_{i=1}^{m} (t_{i})_{p}$. 
\end{Lemma}

Note that, for a finite group $R$, $a(R)$ denotes the composition length of $R$.
\begin{Proposition}[{\bf\cite{GT}, Corollary 5.6}]\label{pq1} Let $R$ be a finite group, let $S$ be a transitive permutation group of degree $s\ge 2$, and let $G$ be a large subgroup in the wreath product $R\wr S$. Then\begin{enumerate}[(i)]
\item If $2\le s\le 1260$, then $d(G)\le \left\lfloor \dfrac{\ws{c}a(R)s}{\log{s}}\right\rfloor +d(\pi(G))$, where $\ws{c}:=2\times 1.25506/\ln{2}=3.621337\hdots$;
\item If $s\ge 1261$, then $d(G)\le \left\lfloor \dfrac{a(R)b_{1}s}{\sqrt{\log{s}}}\right\rfloor +d(\pi(G))$, where $b_{1}:=2/\sqrt{\pi}=1.2838\hdots$.
\end{enumerate}\end{Proposition}
In order to use Proposition \ref{pq1}, we will also need an upper bound on the composition length of a primitive group, in terms of its degree. The bound we require is provided by the next theorem, which is stated slightly differently from how it is stated in \cite{Pyb}.

\begin{Theorem}[{\bf\cite{Pyb}, Theorem 2.10}]\label{pyber} Let $R$ be a primitive permutation group of degree $r\ge 2$, and set $c_{0}:=\log_{9}{48}+\frac{1}{3}\log_{9}{24}=2.24399\hdots$. Then $a{(R)}\le (2+c_{0})\log{r}-(1/3)\log{24}$.\end{Theorem}

Finally, we need the following theorem of Cameron, Solomon and Turull; note that we only give a simplified version of their result here. 
\begin{Theorem}[{\bf\cite{Cam}, Theorem 1}]\label{TransCompLength} Let $G$ be a permutation group of degree $n\ge 2$. Then $a{(R)}\le \frac{3}{2}n$.\end{Theorem}

\subsection{Orders of transitive permutation groups}
We now turn to bounds on the order of a transitive permutation group $G$, of degree $n$. First, we define a function $\bl$ on $G$: If $G$ is primitive, set $R_{1}:=G$ and $r_{1}:=n$. Otherwise, let $r_{1}\ge 2$ denote the size of a minimal block for $G$. Then $G$ is a large subgroup of the wreath product $R_{1}\wr S_{1}$, where $R_{1}$ is primitive of degree $r_{1}$, and $S_{1}$ is transitive of degree $s_{1}:=n/r_{1}$. We can also iterate this process: either $S_{1}$ is primitive, or $S_{1}$ is a large subgroup in a wreath product $R_{2}\wr S_{2}$, where $R_{2}$ is primitive of degree $r_{2}\ge 2$, and $S_{2}$ is transitive of degree $s_{2}:=s_{1}/r_{2}$. Continuing in this way, we see that $G$ is a subgroup in the iterated wreath product $R_{1}\wr R_{2}\wr \hdots\wr R_{t}$, where each $R_{i}$ is primitive of degree $r_{i}$ say, and $\prod_{i} r_{i}=n$. We shall call $W=R_{1}\wr R_{2}\wr \hdots\wr R_{t}$ a \emph{primitive decomposition} of $G$, and the groups $R_{i}$ will be called the \emph{primitive components} of $G$ associated to $W$. Furthermore, we will write $\pi_{i}$ to denote the projection $\pi_{i}:G\le (R_{1}\wr R_{2}\wr \hdots \wr R_{i})\wr (R_{i+1}\wr\hdots\wr R_{t})\rightarrow R_{i+1}\wr\hdots\wr R_{t}$ (for $1\le i \le t-1$).

For each $i$, set $d_i:=r_i$ if $R_i\ge \Alt(r_i)$, and $d_i:=1$ otherwise. Now set $d':=\max_{i}{d_{i}}$, and $d:=\max\left\{d',c_{2}\right\}$, where $c_{2}:=2^{\frac{\log{95040}}{11}}=2.83489\hdots$. Finally, we define $\bl_{W}(G):=d$.

Before proceeding to the main result of this subsection, we require the following theorem of Maroti:
\begin{Theorem}[{\bf\cite{Maroti}, Corollary 1.4}]\label{Maroti} Let $G$ be a primitive permutation group of degree $r$, not containing $\Alt(r)$. Then $|G|\le {c_{2}}^{r-1}$, where $c_{2}:=2^{\frac{\log{95040}}{11}}=2.83489\hdots$.\end{Theorem}

We can now prove the following:
\begin{Proposition}\label{OrderBound} Let $G$ be a transitive permutation group of degree $n$, let $W=R_{1}\wr\hdots\wr R_{t}$ be a primitive decomposition of $G$, where each $R_{i}$ is primitive of degree $r_{i}$. Also, let $d:=\bl_{W}(G)$ be as defined prior to Theorem \ref{Maroti}. Then $|G| \le d^n$.\end{Proposition}
\begin{proof} Working by induction on $n$, the claim follows when $G$ is primitive, since either $d=n$; or $|G| \le {c_{2}}^n$ (by Theorem \ref{Maroti}). So assume that $G$ is imprimitive, and let $r:=r_{1}$, $R:=R_{1}$, $S:=\pi_{1}(G)$. Then $S$ is transitive of degree $s:=n/r$ (being a large subgroup of $R_{2}\wr\hdots\wr R_{t}$), and $G$ is a large subgroup of the wreath product $R\wr S$. Suppose first that $R \ge \Alt(r)$. Then $|R|\le r^{r-1} \le d^{r-1}$. Also, the inductive hypothesis implies that $|S|\le d^s$. Hence, $|G|\le d^{(r-1)s}d^s=d^{rs}$, as needed. So assume that $R$ is not the alternating or symmetric group of degree $r$. Then Theorem \ref{Maroti} implies that $|R| \le {c_{2}}^{r-1} \le d^{r-1}$. The claim now follows, as above, using the inductive hypothesis.\end{proof}

\section{Bounding $d(G)$ in terms of $n$ and $\bl_{W}(G)$}
\begin{Proposition}\label{AltSolTrans} Let $n$ be a positive integer. Then the alternating group $\Alt(n)$ contains a soluble transitive subgroup.\end{Proposition}
\begin{proof} If $n$ is odd, then the group generated by an $n$-cycle suffices, so assume that $n$ is even, and write $n=2^{k}r$, with $r$ odd. Let $P$ be a Sylow $2$-subgroup of $\Alt(2^{k})$, and let $x$ be an $r$-cycle in $\Alt(r)$. Then $P$ is transitive, and the wreath product $P\wr \langle x\rangle$ (in its imprimitive action) is a soluble transitive subgroup of $\Alt(n)$.\end{proof} 

\begin{Proposition}\label{2aProp} Let $H$ be a finite group with a subgroup $H_{1}$ of index $u\ge 2$, let $V$ be a $H_{1}$-module of dimension $a$ over a field $\mathbb{F}$ of characteristic $p>0$, and let $U\le \Sym{(u)}$ be the image of the induced action of $H$ on the set of right cosets of $H_{1}$. If $U\in \left\{\Alt{(u)},\Sym{(u)}\right\}$, then each submodule of the induced module $V\uparrow^H_{H_{1}}$ can be generated by $2a$ elements.\end{Proposition}
\begin{proof} We claim that $U$ contains a soluble subgroup $T$ which has at most two orbits, and each orbit has $p'$-length. To see this, assume first that $p=2$. Then since $n$ is either odd, or a sum of two odd numbers, we can take $T:=\langle x_{1}x_{2}\rangle$, where $x_{1}$ is a cycle of odd length, either $x_{2}=1$ or $x_{2}$ is a cycle of odd length, and $n$ is the sum of the orders (i.e. lengths) of $x_{1}$ and $x_{2}$.

So assume that $p>2$, and write $n=tp+k$, where $0\le k\le p-1$. If $k\neq p-1$, then take $T_{1}$ to be a soluble transitive subgroup of $\Alt(tp-1)$, and take $T_{2}$ to be a soluble transitive subgroup of $\Alt(k+1)$ (the existence of these groups is guaranteed by Proposition \ref{AltSolTrans}). If $k=p-1$, then take $T_{1}$ to be a soluble transitive subgroup of $\Alt(tp+1)$, and take $T_{2}$ to be a soluble transitive subgroup of $\Alt(k-1)$ (note that $k-1>0$ since $p>2$). Finally, taking $T:=T_{1}\times T_{2}\le \Alt(n)$ give us what we need, and proves the claim. 

The result now follows immediately from Lemma \ref{pmodlemma}.\end{proof}

\begin{Lemma}\label{MainLemma} Let $R$ be a finite group, let $U$ and $V$ be permutation groups of degree $u$ and $v$ respectively, and let $S$ be a large subgroup of the wreath product $U\wr V$. Also, let $G$ be a large subgroup of the wreath product $W:=R\wr S$. If $U\in \left\{\Alt{(u)},\Sym{(u)}\right\}$, then $d(G)\le 2a(R)v+d(S)$.\end{Lemma}
\begin{proof} Clearly we may assume that $R$ is nontrivial. Let $B$ denote the base group of $W$, so that $G\cong R^{uv}$. Since $G/G\cap B\cong S\le U\wr V$, we may choose subgroups $H_{1}\le H$ of $G$, containing $G\cap B$, such that $H_{1}/G\cap B$ is a point stabiliser in $S$, and $H/G\cap B$ is the stabiliser of a block $\Delta$ of size $u$ in $S$. Hence, $|G:H_{1}|=uv$ and $|G:H|=v$. 

Since $G/G\cap B\cong S\le U\wr V$ is large, $H^{\Delta}\cong U$. Note also that the permutation action of $S$ corresponds to the action of $S$ on $B$ (by permutation of the direct factors in $B\cong R^{uv}$); hence, since $H/G\cap B$ stabilises a block of size $u$, $H$ normalises a subgroup $B_{1}\cong R^{u}\le B$. In the same way, $H_{1}$ normalises one of the direct factors in $B\cong R^{uv}$: identify $R$ with this direct factor.  

Next, let $L$ be a minimal normal subgroup of $R$, and, viewing $L$ as a subgroup of $R\le B$, let $K$ be the direct product of the distinct $G$-conjugates of $L$. Also, let $K_{1}:=K\cap B_{1}\cong L^{u}$. If $L$ is nonabelian, then $G\cap K$ is either trivial or a minimal normal subgroup of $G$ (see \cite[proof of Lemma 5.1]{GT}), so $d(G)\le 1+d(G/G\cap K)\le a(L)+d(G/G\cap K)$. 

Assume now that $L$ is elementary abelian, of order $p^{a}$ say. Then $K_{1}$ is a $H$-module, generated by the $H_{1}$-module $L$, and $\dim{K_{1}}=u\dim{L}=|H:H_{1}|\dim{L}$. Thus, by \cite[Corollary 3, page 56]{Alp}, $K_{1}$ is isomorphic to the $H$-module induced from the $H_{1}$-module $L$. By a similar argument, $K$, as a $G$-module, is isomorphic to the $G$-module induced from the $H$-module $K_{1}$. Hence, since each $H$-submodule of $K_{1}$ can be generated by $2a$ elements by Proposition \ref{2aProp}, it follows that each $G$-submodule of $K\cong {K_{1}}\uparrow^G_{H}$ can be generated by $2|G:H|a=2va$ elements. Thus, $d(G)\le d_{G}(G\cap K)+d(G/G\cap K)\le 2va(L)+d(G/G\cap K)$.

We are now ready to prove the lemma by induction on $R$: by the previous two paragraphs, in each of the cases of $L$ being abelian or nonabelian we have, in particular, $d(G)\le 2va(L)+d(G/G\cap K)$. If $R=L$ then the result follows, and this can serve as the base step for induction. So assume that $R>L$. Note that $G/G\cap K$ is a large subgroup of the wreath product $R/L \wr S$ satisfying the hypothesis of the lemma, so the inductive hypothesis implies that $d(G/G\cap K)\le 2va(R/L)+d(S)$. The proof is now complete, since $a(R)=a(R/L)+a(L)$.\end{proof}

\section{Proof of Theorem \ref{LimitTheorem}}
We are now ready to prove Theorem \ref{LimitTheorem}.       
\begin{proof}[Proof of Theorem \ref{LimitTheorem}] Let $G$ be a transitive permutation group of degree $n$, and let $W=R_{1}\wr R_{2}\wr\hdots\wr R_{t}$ be a primitive decomposition of $G$, where each $R_{i}$ is primitive of degree $r_{i}$ say. Also, let $d=\bl_{W}(G)$. If $G$ is primitive, then since $d(\Alt{(n)})\le d(\Sym{(n)})\le 2$, Theorems \ref{derekthm} and \ref{Maroti} imply that $f(n)/n^{2}\le 2n\log{n}/n^{2}$, which tends to $0$ as $n$ tends to $\infty$. 

So we may assume that $G$ is imprimitive. Note that if we set $r'=\prod_{j\le k}r_{j}$ and $s'=\prod_{j>k}r_{j}$, for some $1\le k\le t-1$, then Proposition \ref{OrderBound} implies that \begin{align}\log{|G|}\le r's'\log{d}\end{align}
Clearly we may also assume that $n\ge 51 (>4^{c_{2}})$.

Before proceeding, we fix some notation: let $R=R_1$, $S=\pi_{1}(G)\le R_2 \wr\hdots\wr R_t$, $r=r_{1}$, and $s=n/r$, so that $S$ is transitive of degree $s$. Also, if one of the $R_j$ for $j\ge 2$, say $R_i$, is an alternating or symmetric group of degree $d$, then set $\ws{R}:=R_1 \wr R_2 \wr \hdots \wr R_{i-1}$, $\ws{S}:=\pi_{i-1}(G)\le R_i \wr \hdots \wr R_t$, $\ws{r}:=\prod_{j<i} r_j$, and $\ws{s}:=n/\ws{r}$. Otherwise, set $\ws{R}:=R$, $\ws{S}:=S$, $\ws{r}:=r$, and $\ws{s}:=s$. Finally, let $C:=c_{0}+2$, where $c_{0}$ is as in Theorem \ref{pyber}. 

We split the remainder of the proof into two cases: suppose first that either $d=r$ or $d \le \max\left\{\log{\ws{r}},\log{\ws{s}}\right\}$. Then, from the definitions of $r$, $s$, $\ws{r}$ and $\ws{s}$, we see that either $d=r$, or $d\neq r$ and one of the following holds \begin{enumerate}[(a)]
\item $d\le \log{\ws{s}}\le \log{s}$, or;
\item $\log{\ws{s}}<d\le \log{\ws{r}}$. In this case, either $\ws{r}\le s$, in which case $d\le \log{s}$; or $\ws{r}>s$, in which case $r>\ws{s}$. Since either $\ws{s}>d$ or $d=c_{2}$ and $s=\ws{s}=2$ (recall that $d\neq r$), it follows that $\ws{r}>s$ implies that $r>d$.
\end{enumerate}
Now, using Proposition \ref{pq1}, there exists a constant $b_{1}'$ such that
$$d(G)\le  \frac{a(R)b_{1}'s}{\sqrt{\log {s}}}+d(\pi(G))$$
Combining this with Theorems \ref{TransTheoremCor1} and \ref{pyber}, and the inequality at (3.1), we get
\begin{align*} \frac{d(G) \log{|G|}}{n^2} &\le  \frac{(Cb_{1}'\log{r} + c_{1})s}{\sqrt{\log {s}}}\frac{rs\log{d}}{r^{2}s^{2}}\\
                               &= \frac{(Cb_{1}'\log r + c_{1})\log d}{r\sqrt{\log s}}\end{align*}
Since either $d\le r$ or $d\le \log{s}$ (see (a),(b) and the preceding comment above), this goes to $0$ if either $r$ or $s$ is increasing, which gives us what we need.

Finally, assume that $d>\max\left\{\log{\ws{r}},\log{\ws{s}}\right\}$, and that $d\neq r$. Since $n=\ws{r}\ws{s}$ and $n>4^{c_{2}}$, we have $d>c_{2}$, so one of the $R_j$ for $j\ge 2$, must be an alternating or symmetric group of degree $d$. Thus, by definition we have $\ws{R}:=R_1 \wr R_2 \wr \hdots \wr R_{i-1}$, $\ws{S}:=\pi_{i-1}(G)\le R_i \wr \hdots \wr R_t$, and $\ws{r}:=\prod_{j<i} r_j$. Then, by Lemma \ref{MainLemma} and Theorems \ref{TransTheoremCor1} and \ref{TransCompLength}, we have
\begin{align*} d(G) &\le \frac{a(\ws{R})\ws{s}}{d}+d(\ws{S}) &\text{(by Lemma \ref{MainLemma})}\\
&\le  \frac{a(\ws{R})\ws{s}}{d} +\frac{(2b_{1}'+c_{1})\frac{\ws{s}}{d}}{\sqrt{\log {\frac{\ws{s}}{d}}}} &\text{(by Proposition \ref{pq1} and Theorem \ref{TransTheoremCor1})} \\
&\le \frac{\frac{3}{2}\ws{r}\ws{s}}{d} +\frac{(2b_{1}'+c_{1})\frac{\ws{s}}{d}}{\sqrt{\log {\frac{\ws{s}}{d}}}} &\text{(by Theorem \ref{TransCompLength})}\end{align*}
We make a further comment: the second inequality above follows from Proposition \ref{pq1}, since $a(R_{i})\le 2$. Combining the last inequality above with the upper bound at (3.1), we have
\begin{align*} \frac{d(G) \log{|G|}}{n^2} &\le  \left[\frac{\frac{3}{2}\ws{r}\ws{s}}{d} +\frac{(2b_{1}'+c_{1})\frac{\ws{s}}{d}}{\sqrt{\log {\frac{\ws{s}}{d}}}}\right]\frac{\ws{r}\ws{s}\log{d}}{\ws{r}^{2}\ws{s}^{2}}\\
                               &\le \frac{\frac{3}{2}\log{d}}{d} + \frac{(2b_{1}'+c_{1})\log{d}}{\ws{r}d\sqrt{\log{\frac{\ws{s}}{d}}}}\end{align*}
Since $d>\max\left\{\log{\ws{r}},\log{\ws{s}}\right\}$, $d\le \ws{s}$ and $n=\ws{r}\ws{s}$, the result now follows.\end{proof}


\begin{thebibliography}{5}
\bibitem{Alp} Alperin, J.L. \emph{Local Representation Theory}. Cambridge University Press, Cambridge, 1986.
\bibitem{Cam} Cameron, P.J.; Solomon, R.G.; Turull, A. Chains of subgroups in symmetric groups. \emph{J. Algebra} {\bf 127} (1989) 340-352.
\bibitem{derek} Holt, D.F.; Roney-Dougal, C.M. Minimal and random generation of permutation and matrix groups. \emph{J. Algebra} {\bf 387} (2013) 195-223.
\bibitem{Maroti} Maroti, A. On the orders of primitive groups. \emph{J. Algebra} {\bf 258 (2)} (2002) 631-640.
\bibitem{Pyb} Pyber, L. Asymptotic results for permutation groups. \emph{Groups and Computation} DIMACS Ser. Discrete Math. Theoret. Computer Sci. {\bf 11} (ed. Finkelstein, L. and Kantor, W.M., Amer. Math. Soc., Providence, 1993) 197-219.
\bibitem{GT} Tracey, G.M. Minimal generation of transitive permutation groups. \emph{preprint}, available from http://arxiv.org/abs/1504.07506v2.
\end{thebibliography}
\end{document}